\documentclass[10pt]{article}
\font\smallit=cmti10
\font\smalltt=cmtt10

\usepackage{caption}
\usepackage{color}
\usepackage{amssymb,latexsym,amsmath,epsfig,amsthm} 
\usepackage{empheq}
\usepackage{multicol}
\usepackage{here}
\usepackage{url}
\usepackage{ascmac}
\makeatletter
\usepackage{here}
\usepackage{comment}

\renewcommand\section{\@startsection {section}{1}{\z@}
{-30pt \@plus -1ex \@minus -.2ex}
{2.3ex \@plus.2ex}
{\normalfont\normalsize\bfseries\boldmath}}

\renewcommand\subsection{\@startsection{subsection}{2}{\z@}
{-3.25ex\@plus -1ex \@minus -.2ex}
{1.5ex \@plus .2ex}
{\normalfont\normalsize\bfseries\boldmath}}

\renewcommand{\@seccntformat}[1]{\csname the#1\endcsname. }

\makeatother
\newtheorem{theorem}{Theorem}
\newtheorem{lemma}{Lemma}

\theoremstyle{definition}
\newtheorem{definition}{Definition}
\newtheorem{remark}{Remark}

\begin{document}

\begin{center}
\uppercase{\bf Variants of Wythoff's Games with Different Terminal Sets}
\vskip 20pt
{\bf  Kahori Komaki  }\\
{\smallit Keimei Gakuin Junior and High School, Kobe City, Japan}\\
{\tt xiaomukahori@gmail.com}
\vskip 10pt
{\bf Ryohei Miyadera}\\
{\smallit Keimei Gakuin Junior and High School, Kobe City, Japan}\\
{\tt runnerskg@gmail.com}
\vskip 10pt
{\bf Aoi Murakami }\\
{\smallit Kwansei Gakuin University }\\
{\tt atatpj728786.55986@gmail.com}
\end{center}
\vskip 20pt
\centerline{\smallit Received: , Revised: , Accepted: , Published: } 
\vskip 30pt


\centerline{\bf Abstract}
\noindent
We study a variant of the classical Wythoff's game. The classical form is played with two piles of stones, from which two players take turns to remove stones from one or both piles. When removing stones from both piles, an equal number must be removed from each. The player who removes the last stone or stones is the winner.
Equivalently, we consider a single chess queen placed somewhere on a large grid of squares. Each player can move the queen toward the upper-left corner of the grid, either vertically, horizontally, or diagonally in any number of steps. The winner is the player who moves the queen to the terminal position in the upper-left corner, the position $(0,0)$ in our coordinate system. Let $\mathbb{Z}_{\geq0}$ and $\mathbb{N}$ be the set of non-negative integers and the set of positive integers. Let $k \in \mathbb{N}$, and we consider the variant of Wythoff's game with the 
 terminal set $\{(x,y):x,y \in \mathbb{Z}_{\geq0} \text{ and } x+y \leq k\}$. The set of P-positions of this 
 variant is described by the Fibonacci sequence without using recursion.

 \pagestyle{myheadings} 
 \markright{\smalltt \hfill} 
 \thispagestyle{empty} 
 \baselineskip=12.875pt 
 \vskip 30pt

\section{Introduction to Wythoff's Game and its Variant} 
Let $\mathbb{Z}_{\geq0}$ and $\mathbb{N}$ be the sets of non-negative numbers and natural numbers, respectively. 
We study a variant of the classical Wythoff's game. The classical form is played with two piles of stones, from which two players take turns to remove stones from one or both piles. When removing stones from both piles, an equal number must be removed from each. The player who removes the last stone or stones is the winner.
Equivalently, we consider a single chess queen placed somewhere on a large grid of squares. Each player can move the queen toward the upper-left corner of the grid, either vertically, horizontally, or diagonally in any number of steps. The winner is the player who moves the queen to the upper-left corner, position $(0,0)$ in our coordinate system. We call $(0,0)$ the terminal position of Wythoff's game.
In our variant of Wythoff's game, we have a set of positions $\{(x,y):x+y \leq k\}$ for some $k \in \mathbb{N}$ as the terminal set. The player who moves in this terminal set is the winner.

First, we define Wythoff's game. For the details of Wythoff's game, see \cite{wythoffpaper}.

\begin{definition}\label{wythoff}
Wythoff's game is played with two piles of stones. Two players take turns removing stones from one or both piles. When removing stones from both piles, the number of stones removed from each pile should be equal. The player who removes the last stone or stones wins.
An equivalent description of the game is that a single chess queen is placed somewhere on a large grid of squares, and each player can move the queen towards the upper-left corner of the grid, either vertically, horizontally, or diagonally, for any number of steps.  The winner is the player who moves the queen to the upper-left corner.
\end{definition}

Many people have proposed many variants of Wythoff's game, and the author of the present article also presented a variant in \cite{suetsugu2020} and \cite{integer2025}.
We define a new variant of Wythoff's game that generalizes the results of our article \cite{komaki}.
\begin{definition}\label{wythoffvar}
This variant of Wythoff's game is played like the classical Wythoff’s game, with the terminal set 
 $\{(x,y):x+y \leq k\}$ for  $k \in \mathbb{N}$. 
\end{definition}

Figure~\ref{moveofqueen}  shows the moves that the queen can make, and 
Figures ~\ref{chessboard} and \ref{chessboard2} show the terminal positions of the classical Wythoff's game and the game in Definition \ref{wythoffvar} when $k=5$, respectively.

\begin{figure}[H]
\begin{tabular}{ccc}
\begin{minipage}[t]{0.33\textwidth}
\begin{center}
\includegraphics[height=2.cm]{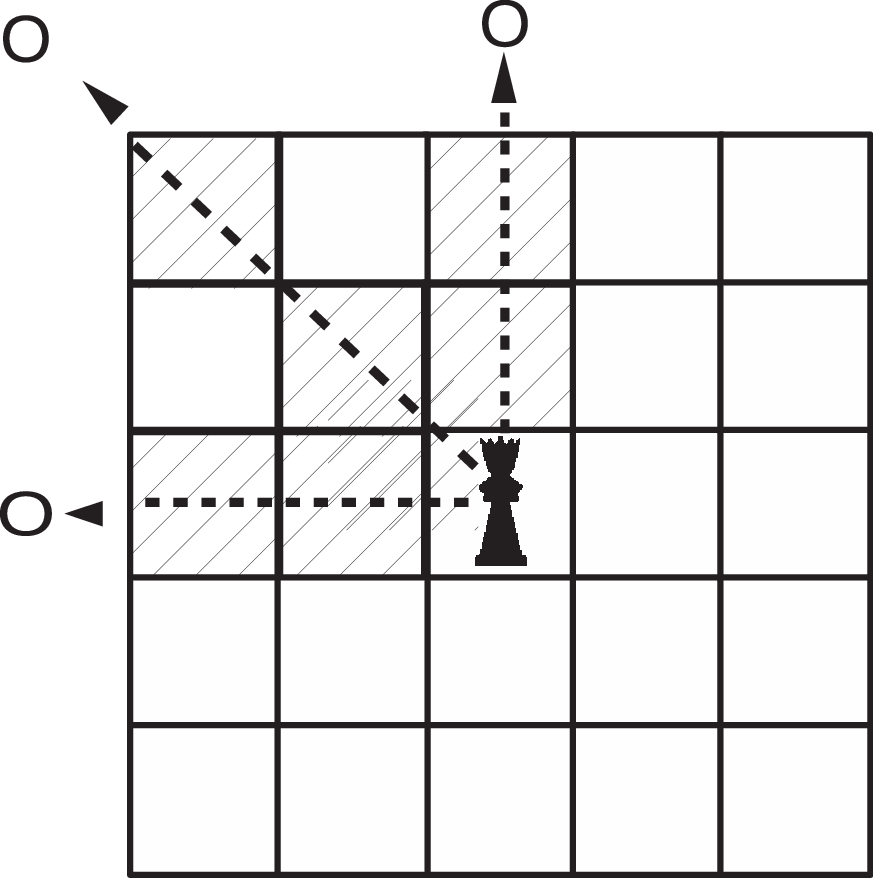}
\captionsetup{labelsep = period}
\caption{The moves of \\ the queen}
\label{moveofqueen}
\end{center}
\end{minipage}
\begin{minipage}[t]{0.33\textwidth}
\begin{center}
	\includegraphics[height=2.5cm]{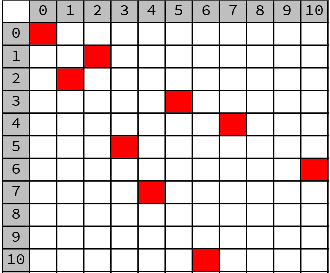}
\captionsetup{labelsep = period}
\caption{The terminal \\ position of Wythoff's \\game}
\label{chessboard}
\end{center}
\end{minipage}
\begin{minipage}[t]{0.33\textwidth}
\begin{center}
\includegraphics[height=2.5cm]{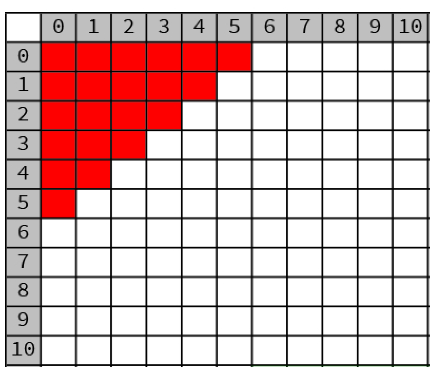}
\captionsetup{labelsep = period}
\caption{The terminal \\ positions of the game in \\ Definition \ref{wythoffvar} with $k=5$}
\label{chessboard2}
\end{center}
\end{minipage}
\end{tabular}
\end{figure}

We define $\textit{move}(x,y)$ in Wythoff's game and its variant in Definition \ref{movewythoff}.
\begin{definition}\label{movewythoff}
$\textit{move}(x,y)$ is the set of all positions that can be reached from $(x,y)$.
 For any $x,y \in Z_{\ge 0}$, let
	\begin{flalign}
	 & \textit{move}(x,y)	=M_1(x,y) \cup M_2(x,y) \cup M_3(x,y),\text{where } & & \nonumber \\
		& M_1(x,y)= \{(u,y):u<x \text{ and } u \in Z_{\ge 0}\}, M_2(x,y)=\{(x,v):v<y \text{ and } v \in Z_{\ge 0}\},& \nonumber \\
		& \text{ and } & \nonumber \\
		& M_3(x,y)=\{(x-t,y-t): 1 \leq t \leq \min(x,y) \text{ and } t \in Z_{\ge 0}\}.& \nonumber
	\end{flalign}
\end{definition}

\begin{remark}\label{explainm1m2m3}
$M_1(x,y), M_2(x,y), M_3(x,y)$ are the sets of horizontal, vertical, and diagonal moves, respectively.
$M_3(x,y)$ is an empty set if $x = 0$ or $y = 0$.
For Wythoff's game $\textit{move}(0,0)= \emptyset$, and for the game in Definition \ref{wythoffvar}, 
$\textit{move}(x,y)= \emptyset$ for $x, y \in \mathbb{Z}_{\geq0}$ such that $x+y \leq k$.
\end{remark}  

\section{Combinatorial Game Theory Definitions and a Theorem}
For completeness, we briefly review some of the necessary concepts in combinatorial game theory by referring to $\cite{lesson}$ and \cite{combysiegel}.

Wythoff's game is an impartial game without drawings; only two outcome classes are possible.
\begin{definition}\label{NPpositions}
A position is referred to as a $\mathcal{P}$-position if it is the winning position for the previous player (the player who has just moved), as long as the player plays correctly at each stage. A position is referred to as an $\mathcal{N}$-position if it is the winning position for the next player, as long as they play correctly at each stage.
\end{definition}

\begin{definition}\label{defofmex}
The \textit{minimum excluded value} ($\textit{mex}$) of a set $S$ of nonnegative integers is the least nonnegative integer that is not in S.
\end{definition}

\begin{definition}\label{defofmexgrundy}
Let $\mathbf{p}$ be a position in the impartial game. The associated \textit{Grundy number} is denoted by $G(\mathbf{p})$ and is 
 recursively defined by 
	$G(\mathbf{p}) = \textit{mex}(\{G(\mathbf{h}): \mathbf{h} \in \textit{move}(\mathbf{p})\}).$
\end{definition}

The next result demonstrates the usefulness of the Sprague--Grundy theory for impartial games.
\begin{theorem}[\cite{lesson}]\label{theoremofsumg}
For the Grundy number in Definition \ref{defofmexgrundy}, we have that 
	$G(\mathbf{p})=0$ if and only if $\mathbf{p}$ is the $\mathcal{P}$-position. \\
\end{theorem}
Using Theorem \ref{theoremofsumg}, we can determine the $\mathcal{P}$-position by calculating the Grundy numbers.
\section{Fibonacci Word}\label{fiboword}
For any lists $A=\{a_1,a_2,\dots, a_n\}$ and $B=\{b_1,b_2,\dots, b_m\}$ we denote
the list $\{a_1,a_2, \dots, a_n, b_1,b_2,\dots, b_m\}$
by $A \uplus B$. For example for $A=\{1,2,3,6\}$ and $B=\{1,3,7,8\}$ 
 $A \uplus B = \{1,2,3,6,1,3,7,8\}$.

Let $S_1 = 0$, $S_2 = 1$, and $S_{n+2} = S_{n+1}S_n$ for $n > 0$;
in other words, $S_{n+2}$ is formed by placing $S_n$ at the right of $S_{n+1}$.
By repeating this process indefinitely, we get a sequence
$0,1,1,0,1,0,1,1,0,1,1,0,\dots $.
By \cite{knuth}, this sequence is the same as the sequence
\begin{equation}
 \lfloor \frac{n+2}{\varphi} \rfloor -  \lfloor \frac{n+1}{\varphi} \rfloor,  \label{knuthseq}  
\end{equation}
for $n \in \mathbb{N}$ and $\varphi= \frac{1+\sqrt{5}}{2}$.

Similarly, starting from the finite sequence $C_{1,1}=\{1\}$, we construct further finite sequences by repeatedly applying the substitution
\begin{equation*}
\sigma(x)=
  \begin{cases}
    2    & \text{if } x=1, \\
    2,1  & \text{if } x=2.
  \end{cases}
\end{equation*}
Then, we get $C_{2,1}=\{2\}$, $C_{3,1}=\{2,1\}$, $C_{4,1}=\{2,1,2\}, \dots $, and 
a sequence
\begin{equation}
 1,2,2,1,2,1,2,2,1,2,1,1,\dots,  \label{ourseq}  
\end{equation}
and by (\ref{knuthseq}), (\ref{ourseq}) is the same as 
\begin{equation}
 \lfloor \frac{n+2}{\varphi} \rfloor -  \lfloor \frac{n+1}{\varphi} \rfloor +1 =  \lfloor (n+2) \varphi \rfloor -  \lfloor (n+1) \varphi \rfloor  \label{ourseq2}  
\end{equation}
for $n \in \mathbb{N}$.
Let $|C_{i,1}|$  be the number of elements in $C_{i,1}$. Let  $F_i$ is the $i$-th term of Fibonacci sequence for $i \in \mathbb{N}$. Then, 
$|C_{i,1}|=F_i$, and for $t \in \mathbb{N}$,
the $t$-th element of $C_{n,1}$ is the $\sum_{i=1}^{n-1}F_i + t = F_{n+1}+t-1$-th term of the sequence
in (\ref{ourseq}). Hence, by (\ref{ourseq2}), the $t$-th element of $C_{n,1}$ is 
\begin{equation}
 \lfloor (F_{n+1}+t+1) \varphi \rfloor -  \lfloor (F_{n+1}+t) \varphi \rfloor.  \label{ourseq3}  
\end{equation}

\section{The Case That $k=5$}
In this section, we consider the case $k=5$; the extension to arbitrary
integers $k \geq 2$ is straightforward.
Starting from the finite sequence $C_1=\{1,1,1,1,1\}$, we construct further finite sequences by repeatedly applying the substitution
$\sigma(x)$ defined in Section \ref{fiboword}.
Thus, from $C_1=\{1,1,1,1,1\}$ we obtain $C_2=\{2,2,2,2,2\}$, and from
$C_2$ we obtain $C_3=\{2,1,2,1,2,1,2,1,2,1\}$.
From $C_3$ we obtain
\[
C_4=\{2,1,2,2,1,2,2,1,2,2,1,2,2,1,2\},\ \dots
\]
Note that $C_i=C_{i,1} \uplus  C_{i,2} \uplus C_{i,3} \uplus C_{i,4} \uplus C_{i,5}$, where
$C_{i,k}=C_{i,1}$ for $k=1,2,3,4,5$.
By concatenating these finite sequences, we define the infinite sequence
$\{c_n : n \in \mathbb{N}\} = C_1 \uplus C_2 \uplus C_3 \uplus \cdots $, namely
$\{c_n\} =$
\[
 \{1,1,1,1,1,\,2,2,2,2,2,\,2,1,2,1,2,1,2,1,2,1,\,
2,1,2,2,1,2,2,1,2,2,1,2,2,1,2,\,\dots\}.
\]

Next, we define the sequence $d_n = c_n + 1$ for $n = 1,2,\dots$,
and set
$D_1=\{2,2,2,2,2\}$, $D_2=\{3,3,3,3,3\}$,
$D_3=\{3,2,3,2,3,2,3,2,3,2\}$, $\dots$, so that
$\{d_n : n \in \mathbb{N}\} = D_1 \uplus D_2 \uplus D_3 \uplus \cdots $.
If we define a further substitution
\begin{equation*}
\rho(x)=
  \begin{cases}
    2    & \text{if } x=2, \\
    2,1  & \text{if } x=3,
  \end{cases}
\end{equation*}
then $\rho(D_i) = C_{i+1}$ for $i=1,2,\dots$.
The substitution $\rho$ is useful when we compare $C_{i+1}$ and $D_i$.

Next, for $n \in \mathbb{N}$, we let
\[
a_n = a_1 + \sum_{i=1}^{n-1} c_i, \qquad a_1 = 6,
\]
and
\[
b_n = b_1 + \sum_{i=1}^{n-1} d_i, \qquad b_1 = 12.
\]
Then the difference sequence of $\{a_n\}_{n \in \mathbb{N}}$ is
$\{c_n\}_{n \in \mathbb{N}}$, and the difference sequence of
$\{b_n\}_{n \in \mathbb{N}}$ is $\{d_n\}_{n \in \mathbb{N}}$.
The construction of $\{a_n\}$ and $\{b_n\}$ is presented in
Tables~\ref{table20n1} and~\ref{table20n2}.

\begin{table}[H] 
  \centering 
  \caption{Sequences $a_n, n=1,2,\dots $ and $c_n, n=1,2,\dots $}
  \label{table20n1} 
{
\setlength{\tabcolsep}{1pt} 
\begin{tabular}{@{\hspace{-3.4pt} \vrule width 0.2pt \hspace{0pt}}r|r|r|r|r|r|r|r|r|r|r|r|r|r|r|r|r|r|r|r|r|r|r|r|r|r|r@{\hspace{0.pt} \vrule width 0.2pt \hspace{0pt}}r|r|r|r|r|r|r|r|r|r|r|r|r|r|r|r|r|r|r|r|r|r|r|r|r|r|r@{\hspace{0.pt} \vrule width 0.2pt \hspace{0pt}}} \hline $a_n$ & 
6 &   & 7 &   & 8 &  & 9 &  & 10 & & 11 & & 13 & & 15 & & 17  &  & 19 & & 21 &  & 23 & & 24 & & 26 & & 27 & $\dots $\\ \hline
& & $\vee$ &  & $\vee$    &   & $\vee$ &  & $\vee$ & & $\vee$ & & $\vee$ &  & $\vee$ & & $\vee$ & &  $\vee$ & & $\vee$ & & $\vee$ & & $\vee$ & & $\vee$ & & $\vee$  &  & $\dots $ \\ \hline
$c_n$ & & 1 &  & 1    &   & 1 &   & 1 &  & 1 & & 2 & & 2 &  & 2 & & 2 & & 2 & & 2 & & 1 & & 2 & & 1 &  & $\dots $ \\ \hline
\end{tabular}
}
\end{table}

\begin{table}[H] 
  \centering 
  \caption{Sequences $b_n, n=1,2,\dots $ and $d_n, n=1,2,\dots $}
  \label{table20n2}        
{
\setlength{\tabcolsep}{1pt} 
\begin{tabular}{@{\hspace{-3.4pt} \vrule width 0.2pt \hspace{0pt}}r|r|r|r|r|r|r|r|r|r@{\hspace{0.4pt} \vrule width 0.2pt \hspace{0pt}}r|r|r|r|r|r|r|r|r|r@{\hspace{0.4pt} \vrule width 0.2pt \hspace{0pt}}} \hline
$b_n$ & 12 &  & 14 &  & 16 &   &  18  &  &  20  &  & 22 &  & 25 & & 28 & $\dots $\\ \hline
&  & $\vee$ &  & $\vee$   &   & $\vee$  &   & $\vee$ & & $\vee$ &  & $\vee$ & & $\vee$ & & $\dots $ \\ \hline
$d_n$ &  & 2 &   & 2   &   & 2 &   & 2  & & 2 & & 3 & & 3 & & $\dots $ \\ \hline
\end{tabular}
}
\end{table}

\begin{lemma}
$(i)$ Let $|C_i|$ and $|D_i|$ be the number of elements in $C_i$ and $D_i$ respectively. Then,
$|C_i|=|D_i|=5F_i$.\\
$(ii)$ The sum of the numbers in $C_i$ is  $\sum \{c_k:c_k \in C_i\}=5F_{i+1}.$\\
$(iii)$ The sum of the numbers in $D_i$ is  $\sum \{d_k:c_k \in D_i\}=5F_{i+2}.$\\
\end{lemma}
\begin{proof}
By the definition of $C_i$ and $D_1$,
we have 
\begin{equation}
C_{i+2}=C_{i} \uplus C_{i+1} \text{ and } D_{i+2}=D_{i} \uplus D_{i+1}.  \label{unionset} 
\end{equation}
Since $|C_1|=|D_1|=|C_2|=|D_2|=5$, by (\ref{unionset}) we obtain (i).
Since  $\sum \{c_k:c_k \in C_1\}=5$ and $\sum \{c_k:c_k \in C_2\}=10$, by (\ref{unionset}) we obtain (ii).
Since $\sum \{d_k:d_k \in D_i\}=\sum \{c_k:c_k \in C_i\} + |D_i|$, by (\ref{unionset}) we obtain (iii).
\end{proof}

\begin{lemma}\label{lemmaforpoints}
For the 
sequences $a_n,b_n$ with $n \in \mathbb{N}$, we obtain the following:\\
$(i)$  $b_n=a_n+n+5$,\\
$(ii)$ $a_{5F_{n+2}-4}=5F_{n+3}-4$;\\
$(iii)$ $b_{5F_{n+1}-4}=5F_{n+3}-3$.
\end{lemma}
\begin{proof}
$(i)$ follows directly from the definitions of $a_n$ and $b_n$.\\
$(ii)$ We consider the $1+\sum_{i=1}^n|C_i|$-th term of the sequence $a_n$.
$1+\sum_{i=1}^n|C_i|=1+\sum_{i=1}^n5F_i= 1+5(F_{n+2}-1)=5F_{n+2}-4$, and 
$a_{5F_{n+2}-4}=6+\sum_{i=1}^n(\sum \{x:x \in C_i\})$ $=6+\sum_{i=1}^{n}5F_{i+1}$
$=6+5(F_{n+3}-2)=5F_{n+3}-4$.\\
$(iii)$ By $(i)$ and $(ii)$, 
$b_{5F_{n+1}-4}=a_{5F_{n+1}-4}+5F_{n+1}-4+5 = 5F_{n+3}-3$.
\end{proof}

\begin{lemma}\label{lemmaforanbmsequence}
We obtain the following statements $(i)$, $(ii)$, $(iii)$, $(iv)$, and $(v)$.\\
(i)  For $p,q \in \mathbb{N}$ such that $5F_{n+2}-4 \leq p < 5F_{n+3}-4$, $5F_{n+1}-4 \leq q < 5F_{n+2}-4$ and $b_q=a_p+1$, we have the following:\\
(i.1) if $b_{q+1}-b_q=3$, then $a_{p+1}=a_p+2$, $a_{p+2}=a_{p}+3$ and $b_{q+1}=a_{p+2}+1$;\\
(i.2) if  $b_{q+1}-b_q=2$, then $a_{p+1}=a_p+2$ and $b_{q+1}=a_{p+1}+1$.\\
(ii) $\{a_{5F_{n+2}-4+u}:u=0,1, \dots, 5F_{n+1}\}$
$\cup \{b_{5F_{n+1}-4+u}:u=0,1,\dots, 5F_{n}\}$\\
$=\{5F_{n+3}-4,\dots, 5F_{n+4}-3\}$.\\
(iii) $\{a_{5F_{n+2}-4+u}:u=0,1, \dots, 5F_{n+1}\}$
$\cap \{b_{5F_{n+1}-4+u}:u=0,1,\dots, 5F_{n}\}= \emptyset$.\\
(iv) $\{a_n,n=1,2,\dots \} \cap \{b_n,n=1,2,\dots \}= \emptyset$; \\
(v) $\{a_n,n=1,2,\dots \} \cup \{b_n,n=1,2,\dots \}= \{6,7,8,\dots \}$.
\end{lemma}
\begin{proof}
We prove $(i),$ $(ii)$ and $(iii)$ by mathematical induction.
Let $D_n=\{d_i,d_{i+1},  $\\
$\dots, d_{i+5F_n-1}\}$ and $C_{n+1}=\{c_j,c_{j+1}, \dots, c_{j+5F_{n+1}-1}\}.$ By the definition of two sequences $a_n$ and $b_n$,
 the difference sequences of $\{a_{5F_{n+2}-4+u}:u=0,1, \dots, 5F_{n+1}\}$ and 
$\{b_{5F_{n+1}-4+u}:u=0,1,\dots, 5F_{n}\}$ are $C_{n+1}$ and $D_n$ respectively, and $\rho(D_n)=C_{n+1}.$
For the brevity of expression, we let $s=5F_{n+2}-4$ and $t=5F_{n+1}-4$.
Then, $b_t=a_s+1.$
Since $d_i=3$ and $\rho(3) = 2, 1$, we obtain $c_j=2$ and $c_{j+1}=1$. Then,
$a_{s+1}=a_s+c_j=a_s+2,$ $a_{s+2}=a_{s+1}+c_{j+1}=a_s+3,$ $b_{t+1}=b_t+d_i=b_t+3=a_{s}+4$.
Here, $a_s ,b_t, a_{s+1}, a_{s+2}, b_{t+1}$ are consecutive integers.
As an inductive step, for $k, h$ such that
$1 \leq k \leq 5F_n-2$, $1 \leq h \leq 5F_{n+1}-2$,
let $D_{n,k}=\{d_i,d_{i+1}, \dots, d_{i+k}\}$ and $C_{n+1,h}=\{c_j,c_{j+1}, \dots, c_{j+h}\}.$

We assume that $\rho(D_{n,k})=C_{n+1,h}$,
the set $\{a_s,a_{s+1},\dots, a_{s+h+1}\} \cup $ \\
$ \{b_t,b_{t+1}, \dots, b_{t+k+1}\}$ constitutes a set of  
consecutive integers, and $a_{s+h+1}+1=b_{t+k+1}.$

For the inductive step, we need to treat two cases. \\
\noindent {\tt Case 1}: Suppose that $d_{i+k+1}=3$. Then, $\rho(3)=2,1$ and 
$c_{j+h+1}=2$ and $c_{j+h+2}=1$. Hence, $a_{s+h+1}$, $b_{t+k+1}$, 
$a_{s+h+2}=a_{s+h+1}+2$, $a_{s+h+3}=a_{s+h+2}+1=a_{s+h+1}+3$, and 
$b_{t+k+2}=b_{t+k+1}+d_{i+k+1}=a_{s+h+1}+4$ are consecutive integers.
let $D_{n,k+1}=\{d_i,d_{i+1}, \dots, d_{i+k+1}\}$ and $C_{n+1,h+2}=\{c_j,c_{j+1}, \dots, c_{j+h+2}\}.$
Then, $\rho(D_{n,k+1})=C_{n+1,h+2}$ and $b_{t+k+2}=a_{s+h+3}+1.$\\
\noindent {\tt Case 2}: Suppose that $d_{i+k+1}=2$. Then, $\rho(2)=2$ and 
$c_{j+h+1}=2$. Hence, $a_{s+h+1}$, $b_{t+k+1}=a_{s+h+1}+1$,
$a_{s+h+2}=a_{s+h+1}+2$,  and 
$b_{t+k+2}=b_{t+k+1}+2=a_{s+h+1}+3$ are consecutive integers.
Let $D_{n,k+1}=\{d_i,d_{i+1}, \dots, d_{i+k+1}\}$ and $C_{n+1,h+1}=\{c_j,c_{j+1}, \dots, c_{j+h+1}\}.$
Then, $\rho(D_{n,k+1})=C_{n+1,h+1}$ and $b_{t+k+2}=a_{s+h+2}+1$.
Therefore, by mathematical induction, we obtain $(i)$, $(ii)$ and $(iii)$, and 
hence, we obtain $(iv)$ and $(v)$.
\end{proof}

Starting from the finite sequence $C_1=\{1\}$, we construct further finite sequences by repeatedly applying the substitution
\begin{equation*}
\sigma(x)=
  \begin{cases}
    2    & \text{if } x=1, \\
    2,1  & \text{if } x=2.
  \end{cases}
\end{equation*}

\begin{theorem}\label{bneqann5}
We obtain the following:\\
$(i)$
\begin{align}
a_{5(F_{n+1}-1)+hF_n+t+1} 
 = 5F_{n+2}-4+hF_{n+1}  \nonumber + \lfloor (F_{n+1}+1+t) \varphi \rfloor  -\lfloor (F_{n+1}+1) \varphi \rfloor  \nonumber  
\end{align}
for $h,t \in \mathbb{Z}_{\geq0}$ such that $0 \leq h \leq 4$ and $1 \leq t \leq F_n$,\\
$(ii)$
\begin{align}
b_{5(F_{n+1}-1)+hF_n+t+1}
 = 
 5F_{n+3}+hF_{n+2} -3+t  + \lfloor (F_{n+1}+1+t) \varphi \rfloor  -\lfloor (F_{n+1}+1) \varphi \rfloor  \nonumber  
\end{align}
for $h,t \in \mathbb{Z}_{\geq0}$ such that $0 \leq h \leq 4$ and $1 \leq t \leq F_n$.
\end{theorem}
\begin{proof}
$(i)$ We consider $a_n$ for $n=5(F_1+\cdots + F_{n-1})+hF_n + t+1 = 5(F_{n+1}-1)+hF_n+t+1$. 
To calculate $a_n$, we need to use $a_1=6,$ $C_1$, $C_2$, $\dots, C_{n-1},$ $C_{n,1},\dots, C_{n,h}$ and the first $t$
 elements of $C_{n,h+1}$, where for $k=1,2,3,4,5$, $C_{n,k}=C_{n,1}$ defined in Section \ref{fiboword}.
By (\ref{ourseq3}), the first $t$ elements of $C_{n,h+1}=C_{n,1}$ are
\begin{equation}
\lfloor (F_{n+1}+i+1) \varphi \rfloor  -\lfloor (F_{n+1}+i) \varphi \rfloor \nonumber
\end{equation}
for $i=1,2, \dots, t$, and the sum of these elements is
\begin{equation}
\lfloor (F_{n+1}+t+1) \varphi \rfloor  -\lfloor (F_{n+1}+1) \varphi \rfloor. \label{fn1t1}
\end{equation}
By (\ref{fn1t1}), 
the sum of the difference sequence of $a_1=6, a_2, \dots, a_n$ is 
$\sum_{i=1}^{n-1} \sum\{c_m:c_m \in  C_i \} $    $+ \sum_{i=1}^{h} \sum\{c_m:c_m \in C_{n,i} \} $
$+\lfloor (F_{n+1}+1+t) \varphi \rfloor  -\lfloor (F_{n+1}+1) \varphi \rfloor$
$ = 5(F_2+\cdots + F_n)+hF_{n+1}$$+\lfloor (F_{n+1}+1+t) \varphi \rfloor  -\lfloor (F_{n+1}+1) \varphi \rfloor$
$=5(F_{n+2}-2)+hF_{n+1}+$ $\lfloor (F_{n+1}+1+t) \varphi \rfloor  -\lfloor (F_{n+1}+1) \varphi \rfloor$.
Since $a_1=6$, we obtain $(i).$\\
$(ii)$ is direct from $(i)$ and $(i)$ of Lemma \ref{lemmaforpoints}.
\end{proof}

\begin{definition}\label{wythoffpp}
Let 
\begin{equation}
P_{5,0}=\{(x,y):x,y \in \mathbb{Z}_{\geq0} \text{ and } x+y \leq 5 \}, \nonumber    
\end{equation}
\begin{equation}
P_{5,1}=\{(b_n,a_n ):n \in \mathbb{N} \}, \nonumber    
\end{equation}
\begin{equation}
P_{5,2}= \{(a_n,b_n):n \in \mathbb{N} \}, \nonumber    
\end{equation}
and
\begin{equation}
P_5=P_{5,0} \cup P_{5,1} \cup P_{5,2}. \nonumber    
\end{equation}
\end{definition}

\begin{lemma}\label{twopattern}
For sequences $a_n, b_n \in \mathbb{N}$, we have the following:\\
$(i)$ $a_{n-1}-a_n=1$ or $2$;\\
$(ii)$ if $a_{n-1}-a_n=1$, then $b_{n-1}-b_n=2$;\\
$(ii)$ if $a_{n-1}-a_n=2$, then $b_{n-1}-b_n=3$.
\end{lemma}
\begin{proof}
Since $a_{n-1}-a_n=c_n$ and $c_n=1,2$, by Lemma \ref{lemmaforpoints}, we obtain (i), (ii), and (iii).
\end{proof}

\begin{lemma}\label{positionsofpposition}
$(i)$ For $\mathcal{P}$-positions $(x,y)$ and $(z,w)$, if $x=z$, then $y=w$, and if $y=w$, then $x=z$.\\
$(ii)$ For any $c \in \mathbb{Z}_{\geq0}$, there exist  $x,y  \in \mathbb{Z}_{\geq0}$ such that $(c,y)$ and $(x,c)$ are $\mathcal{P}$-positions.
\end{lemma}
\begin{proof}
$(i)$ If $x=z$ and $y \ne w$, then we can move from $(x,y)$ to $(z,w)$ or $(z,w)$ to $(x,y)$ by $M_1$. Since both of $(x,y),(z,w)$ are $\mathcal{P}$-positions, this leads to a contradiction. Hence $y=w$.
Similarly, if $y=w$, then $x=z$.\\
$(ii)$ Let $c \in \mathbb{Z}_{\geq0}$.
Let $Q=\{(x,y): y \leq c \text{ and } (x,y) \text{ is a $\mathcal{P}$-position}\}.$
By $(i)$, the number of elements in $Q$ is finite. Let $d \in \mathbb{Z}_{\geq0}$ be a sufficiently large number such that $(d,c) \notin Q$ and $M_2(d,c) \cap Q = M_3(d,c) \cap Q = \emptyset$.
Since $(d,c)$ is a $\mathcal{N}$-position, $M_1(d,c) \cap Q \ne \emptyset$.
Therefore, there exists $x \in \mathbb{Z}_{\geq0}$ such that $(x,c)$ is a $\mathcal{P}$-position.
Similarly we can prove that there exists $y \in \mathbb{Z}_{\geq0}$ such that $(c,y)$ is a $\mathcal{P}$-position.
\end{proof}

\begin{lemma}\label{cannotbym3}
Let $h \in \mathbb{Z}_{\geq0}$.
$(i)$ If we start with a position $(b_n+h,a_n)$, where $(b_n,a_n) \in P_{5,1}$,  we cannot reach a position $(x,y) \in P_5$ by the diagonal move $M_3$.\\
$(ii)$ If we start with a position 
 $(a_n,b_n+h)$, where $(a_n,b_n) \in P_{5,12}$, we cannot reach a position $(x,y) \in P_5$ by the diagonal move $M_3$.
\end{lemma}
\begin{proof}
Let $h \in \mathbb{Z}_{\geq0}$.
Suppose that we start with a position $(b_n+h,a_n)$, where  $(b_n,a_n) \in P_{5,1}$.
We cannot reach a position $(b_i,a_i)$ with $i<n$ by $M_3$, since 
$\frac{a_n-a_i}{b_n+h-b_i}= \frac{c_i+\dots + c_{n-1}}{d_i+\dots + d_{n-1}+h}$
$=\frac{c_i+\dots + c_{n-1}}{c_i+\dots + c_{n-1}+n-i+h}<1$.
It is clear that we cannot move to position $(a_m,b_m) \in P_{5,2}$ with $a_m < b_m$ by $M_3$,
since $\frac{a_n-b_m}{b_n+h-a_m}<1$.

Suppose that we start with a position $(a_m,b_m+h) \in P_{5,2}$. We can use a method similar to the one used 
for $(b_n+h,a_n)$ to prove that we cannot reach a position $(x,y) \in P_5$ via the diagonal move $M_3$.
\end{proof}

\begin{theorem}
The set of $\mathcal{P}$-positions of this variant of Wythoff's game is  $P_5$.
\end{theorem}
\begin{proof}
If $x \leq 5$ or $y \leq 5$, we can move from $(x,y)$
to the set  of terminal set $\{(x,y):x+y \leq 5\}$ by
$M_2$ or $M_1$. Therefore, there is no $\mathcal{P}$-position $(x,y)$ such that $x \leq 5$ or $y \leq 5$.

\begin{figure}[H]
\begin{tabular}{cc}
\begin{minipage}[t]{0.5\textwidth}
\begin{center}
\includegraphics[height=2.3cm]{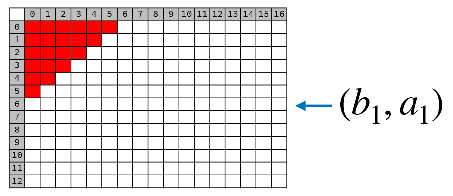}
\captionsetup{labelsep = period}
\caption{Where is $(b_1,a_1)$ ?}
\label{nearte1}
\end{center}
\end{minipage}
\begin{minipage}[t]{0.5\textwidth}
\begin{center}
\includegraphics[height=2.3cm]{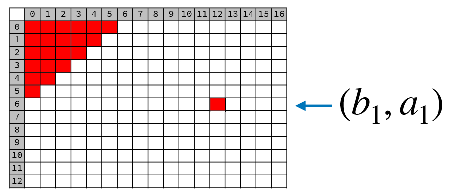}
\captionsetup{labelsep = period}
\caption{the position of $(b_1,a_1)$ }
\label{nearte2}
\end{center}
\end{minipage}
\end{tabular}
\end{figure}

By Lemma \ref{positionsofpposition}, there should be $x_1 \in \mathbb{Z}_{\geq0}$
such that  $(x_1,6)$ is a $\mathcal{P}$-position.
If $6 \leq x_1 \leq 11$, from $(x_1,6)$ we 
 can move to a position in 
$\{(x,y):x+y \leq 5\}$ by $M_3$, as you can see in Figure \ref{nearte1}.  Suppose that 
$x_1 \geq 13$. Then, $(12,6)$ is a $\mathcal{N}$-position and 
from $(12,6)$, we cannot move to any $\mathcal{P}$-position. This leads to a contradiction.
Therefore, $x_1=12.$ Then, 
$(12,6)$ is a  $\mathcal{P}$-position. By Tables \ref{table20n1} and \ref{table20n2}, 
$(12,6)=(b_1,a_1)$. See Tables \ref{nearte2}.
By Lemma \ref{positionsofpposition}, there should be $x_i \in \mathbb{Z}_{\geq0}$
such that  $(x_i,i+5)$ is a $\mathcal{P}$-position for $i=2,3,4,5,6$, and by the method that is similar to the one used to obtain that $x_1=12$, we establish that $x_2=14, x_3=16, x_4=18, x_5=20, x_6=22$. 
By the symmetrical nature of the game up to the first and the second coordinates, we know that
$(6,12),(7,14)$ are $\mathcal{P}$-positions. Here, we use mathematical induction starting from the $\mathcal{P}$-positions in Figure \ref{pvqueen0}. Note that from the set of positions 
$Q_{1,1}=\{(x,11):6 \leq x \leq 21\}$ $\cup \{(x,12):7 \leq x \leq 23\}$, we can move to positions in $P_5$ by the diagonal move $M_3$. The set $Q_{1,1}$ is denoted by the green-colored area in Figure \ref{pvqueen0}.

\begin{figure}[H]
\begin{center}
\includegraphics[height=4.cm]{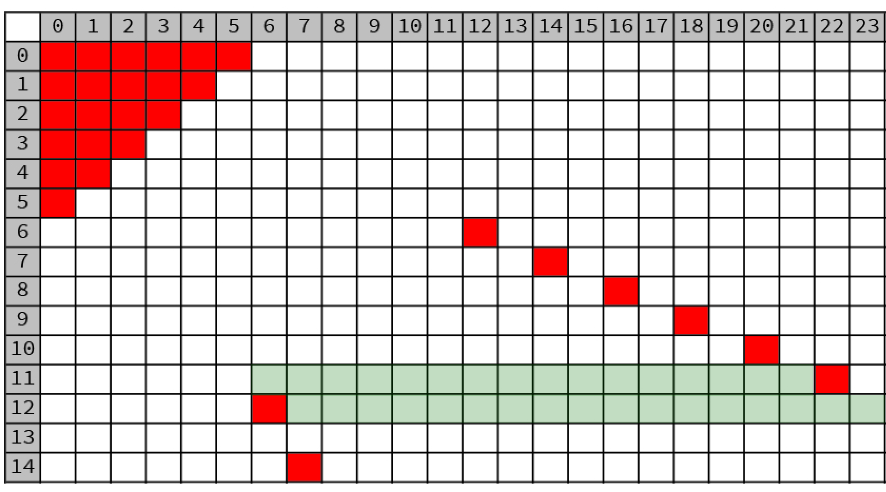}
\captionsetup{labelsep = period}
\caption{$\mathcal{P}$-positions near the terminal set}
\label{pvqueen0}
\end{center}
\end{figure}

\begin{figure}[H]
\begin{center}
\includegraphics[height=6.cm]{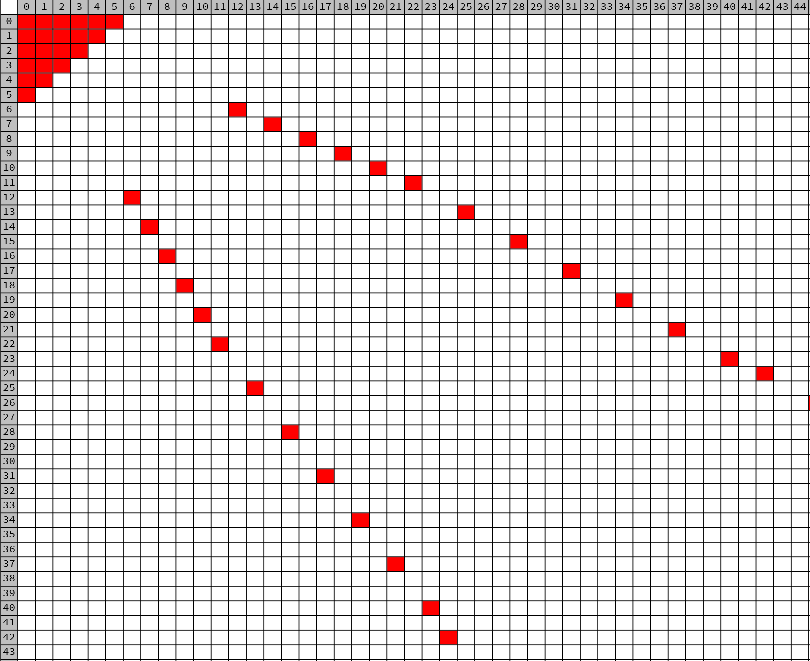}
\captionsetup{labelsep = period}
\caption{$P_5$ }
\label{pvqueen}
\end{center}
\end{figure}
For the mathematical induction step, we assume that the positions in $\{(x,y):x+y \leq 5\}$ $\cup \{(b_i,a_i):i=1,2,\dots, n-1\}$ $\cup \{(a_i,b_i):i=1,2,\dots, n-1\}$
are $\mathcal{P}$-positions, and $b_{m-1}=a_{n-1}+1$, where $m<n$.
We also assume that from the set of positions
$Q_{n-1,m-1}=\{(x,a_{n-1}):a_{m-1}\leq x \leq b_{n-1}-1\}$
$\cup \{(x,b_{m-1}):a_{m-1}+1 \leq x \leq b_{n-1}+1\}$,
we can move to $P_5$ by the diagonal move $M_3$. We present the set $Q_{n-1,m-1}$ as the green area in Figure \ref{inductiongraph}.
\begin{figure}[H]
\begin{center}
\includegraphics[height=1.3cm]{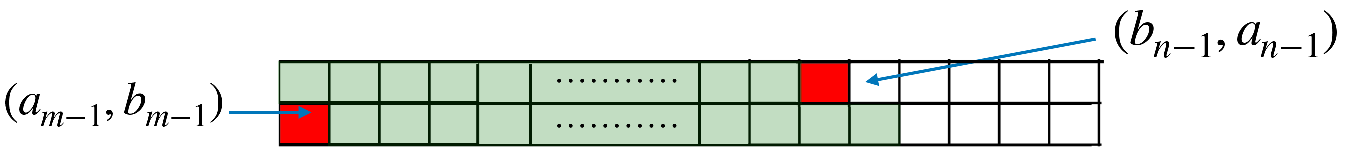}
\captionsetup{labelsep = period}
\caption{$Q_{n-1,m-1}$   }
\label{inductiongraph}
\end{center}
\end{figure}

We aim to prove that $(a_n,b_n)$ is a $\mathcal{P}$-position.
By Lemma \ref{lemmaforpoints},
\begin{equation}
b_{m-1}=a_{m-1}+m+4 \label{plusab2}
\end{equation}
and
\begin{equation}
b_m=a_{m}+m+5. \label{plusab3}  
\end{equation}
By Lemma \ref{twopattern}, 
we have 
\begin{equation}
b_m=b_{m-1}+2 \label{plus2}
\end{equation}
or
\begin{equation}
b_m=b_{m-1}+3. \label{plus3}  
\end{equation}
For $(\ref{plus2})$ and $(\ref{plus3})$, we need to prove two cases. \\
\noindent {\tt Case 1}:     
We suppose $(\ref{plus2})$. Then, by $(\ref{plusab2})$ and $(\ref{plusab3})$,
$a_m=a_{m-1}+1$. Hence, we have positions as in Figures \ref{point2} and \ref{point3}.
Note that $m<n-1$ and $b_{m-1}=a_{n-1}+1$.
By Lemmas \ref{lemmaforanbmsequence} and  \ref{positionsofpposition}, there should be a $\mathcal{P}$-position $(x,a_{n})$ with 
\begin{equation}
a_{n}=b_{m-1}+1=a_{n-1}+2.\label{anbm1}    
\end{equation}
Then, a $\mathcal{P}$-position $(x, a_{n})$ should be in Positions 1,2,3,4,5 or further right in Figure \ref{point3}.
If a $\mathcal{P}$-position $(x,a_{n})$ is in Position 1, then by the diagonal move $M_3$
you can move into $Q_{n-1,m-1}$. By the hypothesis, from a point in 
$Q_{n-1,m-1}$, you can move to a position in $P_5$ that is a $\mathcal{P}$-position by $M_3$.
Therefore, if $(x, a_{n}) \in P$ is in Positions 1, you can move to
 a $\mathcal{P}$-position by the diagonal move $M_3$. This contradicts the fact that 
 $(x,a_{n})$ is a $\mathcal{P}$-position.
If a $\mathcal{P}$-position $(x,a_{n})$ is in Position 2, you can move to the $\mathcal{P}$-position $(b_{n-1},a_{n-1})$ by $M_3$. 
If $(x, a_{n}) \in P$ is in Positions 4, 5, or further right, Position 3 should be an $\mathcal{N}$-position. 
By the diagonal move, we reach $(b_{n-1}+1, a_{n-1})$, but by Lemma \ref{cannotbym3}, we cannot move to a P-position from here by the diagonal move. Since  $x > b_{n-1}$, it is clear that we cannot move to a $\mathcal{P}$-position by 
$M_2$. This leads to a contradiction.
Therefore, $(x,a_{n})$ is in Positions 3 as presented in Figure \ref{point1}, and 
\begin{equation}
x=b_{n-1}+3.\label{syn13}    
\end{equation}
By Lemma \ref{lemmaforpoints},
$b_n=a_n+n+5$ and $b_{n-1}=a_{n-1}+n+4$, and 
hence by (\ref{anbm1}), we obtain $b_n=b_{n-1}+3$.
Therefore, by (\ref{syn13}), we establish that $x=b_n$. We prove that 
$(b_n,a_n)$ is a $\mathcal{P}$-position.
\begin{figure}[H]
\begin{tabular}{cc}
\begin{minipage}[t]{0.5\textwidth}
\begin{center}
	\includegraphics[height=1.4cm]{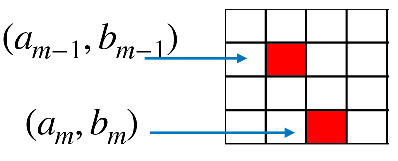}
\captionsetup{labelsep = period}
\caption{ }
\label{point2}
\end{center}
\end{minipage}
\begin{minipage}[t]{0.5\textwidth}
\begin{center}
	\includegraphics[height=1.43cm]{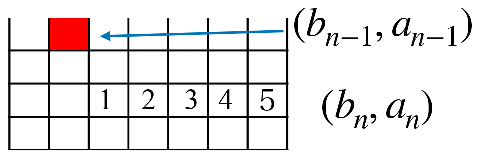}
\captionsetup{labelsep = period}
\caption{ }
\label{point3}
\end{center}
\end{minipage}
\end{tabular}
\end{figure}

\begin{figure}[H]
\begin{center}\includegraphics[height=1.7cm]{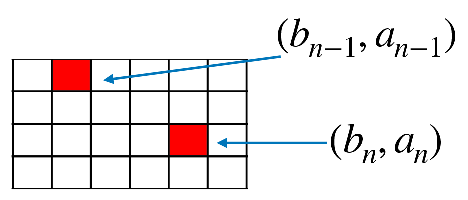}
\captionsetup{labelsep = period}
\caption{ }
\label{point1}
\end{center}
\end{figure}
From the set of positions
$Q_{n,m}=\{(x,a_{n}):a_{m}\leq x \leq b_{n}-1\}$
$\cup \{(x,b_{m}):a_{m}+1 \leq x \leq b_{n}+1\}$,
we can directly move to $P_5$ by the diagonal move $M_3$ or
go into $Q_{n-1,m-1}$ by the diagonal move $M_3$.

\begin{figure}[H]
\begin{center}\includegraphics[height=1.7cm]{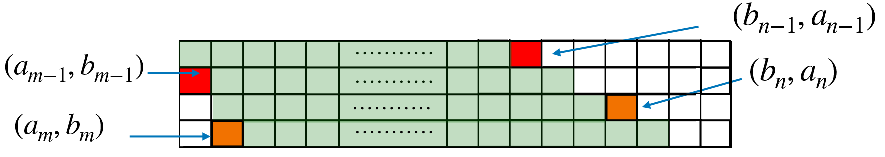}
\captionsetup{labelsep = period}
\caption{ }
\label{inductionlast1}
\end{center}
\end{figure}

\noindent {\tt Case 2}:   
We suppose $(\ref{plus3})$. Then, by $(\ref{plusab2})$ and $(\ref{plusab3})$,
$a_m=a_{m-1}+2$. Hence, the positions are shown in Figures \ref{ppoint1} and \ref{ppoint2}.
By Lemmas \ref{lemmaforanbmsequence} and  \ref{positionsofpposition}, 
there should be a $\mathcal{P}$-position $(x,a_{n})$. Here,
\begin{equation}
a_{n}=b_{m-1}+1=a_{n-1}+2. \label{anbm2}   
\end{equation}
By the method that is similar to the one used in Case 1, we obtain that 
$x=b_{n-1}+3$. The $\mathcal{P}$-position $(x,a_{n})$ is presented in Figure \ref{ppoint3}.
By Lemma \ref{lemmaforpoints},
$b_n=a_n+n+5$ and $b_{n-1}=a_{n-1}+n+4$, and 
hence by (\ref{anbm2}), we obtain $b_n=b_{n-1}+3$.
Therefore, $x=b_n$, and we prove that
$(b_{n},a_n)$ is a $\mathcal{P}$-position. By Lemmas \ref{lemmaforanbmsequence} and  \ref{positionsofpposition}, 
there should be a $\mathcal{P}$-position $(x,a_{n+1})$. Here, $a_{n+1}+1=b_{m}$. By the method that is similar to the one used in Case 1, we obtain that $x=b_{n+1}$, and we prove that 
The position of $(b_{n+1},a_{n+1})$ is a $\mathcal{P}$-position.

\begin{figure}[H]
\begin{tabular}{cc}
\begin{minipage}[t]{0.5\textwidth}
\begin{center}
	\includegraphics[height=1.65cm]{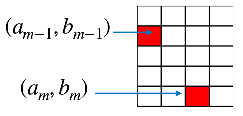}
\captionsetup{labelsep = period}
\caption{ }
\label{ppoint1}
\end{center}
\end{minipage}
\begin{minipage}[t]{0.5\textwidth}
\begin{center}
	\includegraphics[height=1.78cm]{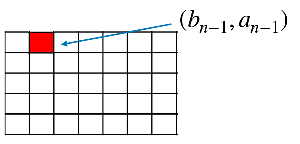}
\captionsetup{labelsep = period}
\caption{ }
\label{ppoint2}
\end{center}
\end{minipage}
\end{tabular}
\end{figure}

\begin{figure}[H]
\begin{tabular}{cc}
\begin{minipage}[t]{0.5\textwidth}
\begin{center}
	\includegraphics[height=1.87cm]{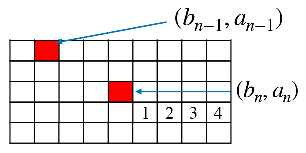}
\captionsetup{labelsep = period}
\caption{Case 1-1}
\label{ppoint3}
\end{center}
\end{minipage}
\begin{minipage}[t]{0.5\textwidth}
\begin{center}
	\includegraphics[height=1.87cm]{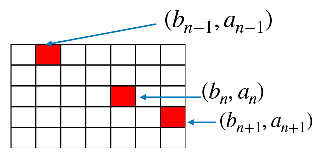}
\captionsetup{labelsep = period}
\caption{Case 1-1}
\label{ppoint4}
\end{center}
\end{minipage}
\end{tabular}
\end{figure}

From the set of positions
$Q_{n+1,m}=\{(x,a_{n}+1):a_{m}\leq x \leq b_{n+1}-1\}$
$\cup \{(x,b_{m}):a_{m}+1 \leq x \leq b_{n+1}+1\}$,
we can directly move to $P_5$ by the diagonal move $M_3$ or
go into $Q_{n-1,m-1}$ by the diagonal move $M_3$.
\begin{figure}[H]
\begin{center}\includegraphics[height=2.13cm]{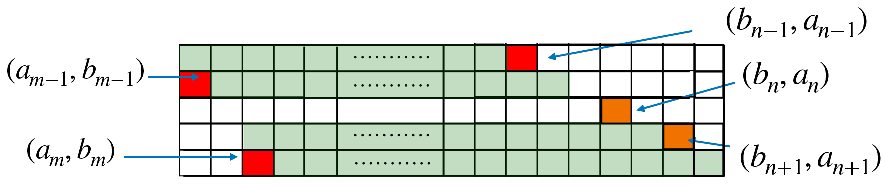}
\captionsetup{labelsep = period}
\caption{ }
\label{inductionlast1}
\end{center}
\end{figure}
By Cases 1 and 2, we completed the proof by mathematical induction.
\end{proof}

\section{Generalization}
In the previous sections, we considered the case when the terminal set is $\{(x,y):x+y \leq 5\}$; however, we can generalize this case to the case that the terminal set is $\{(x,y):x+y \leq k\}$ with a natural number $k$
 without any difficulty.

Let $C_{i}^{(k)}=\biguplus_{j=1}^k C_{i,j}$, where $C_{i,j}=C_{i,1}$ defined in Section \ref{fiboword}.
We define the infinite sequence
$\{c_n : n \in \mathbb{N}\} = C_1^{(k)} \uplus C_2^{(k)} \uplus C_3^{(k)} \uplus \cdots $, and we define the sequence $d_n = c_n + 1$ for $n = 1,2,\dots$.

Next, for $n \in \mathbb{N}$, we let
\[
a_n = a_1 + \sum_{i=1}^{n-1} c_i, \qquad a_1 = k+1,
\]
and
\[
b_n = b_1 + \sum_{i=1}^{n-1} d_i, \qquad b_1 = 2k+2.
\]
Then the difference sequence of $\{a_n\}_{n \in \mathbb{N}}$ is
$\{c_n\}_{n \in \mathbb{N}}$, and the difference sequence of
$\{b_n\}_{n \in \mathbb{N}}$ is $\{d_n\}_{n \in \mathbb{N}}$.

\begin{theorem}\label{bneqannkk}
We obtain the following:\\
$(i)$
\begin{align}
a_{k(F_{n+1}-1)+hF_n+t+1} 
 = kF_{n+2}-k+1+hF_{n+1}  \nonumber + \lfloor (F_{n+1}+1+t) \varphi \rfloor  -\lfloor (F_{n+1}+1) \varphi \rfloor  \nonumber  
\end{align}
for $h,t \in \mathbb{Z}_{\geq0}$ such that $0 \leq h \leq 4$ and $1 \leq t \leq F_n$,\\
$(ii)$
\begin{align}
b_{k(F_{n+1}-1)+hF_n+t+1}
 = 
 kF_{n+3}+hF_{n+2} -k+t +2 + \lfloor (F_{n+1}+1+t) \varphi \rfloor  -\lfloor (F_{n+1}+1) \varphi \rfloor  \nonumber  
\end{align}
for $h,t \in \mathbb{Z}_{\geq0}$ such that $0 \leq h \leq k-1$ and $1 \leq t \leq F_n$.
\end{theorem}
The proof of this theorem is almost the same as that for $k=5$. We omit the proof.
\begin{definition}\label{wythoffppk}
Let 
\begin{equation}
P_{k,0}=\{(x,y):x,y \in \mathbb{Z}_{\geq0} \text{ and } x+y \leq k \}, \nonumber    
\end{equation}
\begin{equation}
P_{k,1}=\{(b_n,a_n ):n \in \mathbb{N} \}, \nonumber    
\end{equation}
\begin{equation}
P_{k,2}= \{(a_n,b_n):n \in \mathbb{N} \}, \nonumber    
\end{equation}
and
\begin{equation}
P_k=P_{k,0} \cup P_{k,1} \cup P_{k,2}. \nonumber    
\end{equation}
\end{definition}

\begin{theorem}
The set of $\mathcal{P}$-positions of this variant of Wythoff's game is  $P_k$.
\end{theorem}
The proof of this theorem is almost the same as that for $k=5$. We omit the proof.

\end{document}